\documentclass[12pt]{article}
\usepackage{amsmath,amsthm,amssymb,euscript,verbatim}
\newtheorem{theorem}{Theorem}[section]

\newtheorem{corollary}{Corollary}[section]
\newtheorem{remark}{Remark}[section]
\newtheorem{example}{Example}[section]
\newtheorem{conjecture}{Conjecture}[section]

\usepackage{makeidx}         
\usepackage{graphicx}        
\usepackage{multicol}        
\usepackage[bottom]{footmisc}

\makeindex             
\usepackage{latexsym}
\usepackage{epic,eepic,pstricks}

\begin{document}
\title
{\bf  Compact Dupin Hypersurfaces}
\author
{Thomas E. Cecil}
\maketitle

\begin{abstract}
A hypersurface $M$ in ${\bf R}^n$ is said to be {\em Dupin} if along
each curvature surface, the corresponding principal curvature is constant.  A Dupin
hypersurface is said to be {\em proper Dupin}
if the number of distinct principal curvatures is constant on $M$, i.e.,
each continuous principal curvature function has constant multiplicity on $M$.
These conditions are preserved by stereographic projection, so this theory is essentially the same for hypersurfaces in ${\bf R}^n$ or $S^n$.
The theory of compact proper Dupin hypersurfaces  in $S^n$
is closely related to the theory of isoparametric  hypersurfaces in $S^n$, and
many important results in this field concern relations between these two classes of hypersurfaces.
In 1985, Cecil and Ryan \cite[p. 184]{CR7} conjectured
that every compact, connected proper Dupin hypersurface
$M \subset S^n$ is equivalent to an isoparametric hypersurface in $S^n$
by a Lie sphere transformation.  This paper
gives a survey of progress on this conjecture and related developments.
\end{abstract}

\section{Dupin hypersurfaces}
\label{sec:1}
This paper is a survey of the main results in the theory of compact, connected proper Dupin hypersurfaces
 in Euclidean space ${\bf R}^n$, which
 began with the study of the cyclides of Dupin in ${\bf R}^3$ in a book by Dupin \cite{D} in 1822.
 This theory is closely related to the well-known theory of isoparametric  hypersurfaces in $S^n$, i.e., 
 hypersurfaces with constant principal curvatures in $S^n$, introduced by E. Cartan
 \cite{Car2}--\cite{Car5} and developed by many mathematicians (see 
 \cite{Cec9},  \cite{Chi-survey}, \cite{Th6}, \cite[pp. 85--184]{CR8} for surveys).  
 
 In fact, many results in this field
 concern conditions under which a compact proper Dupin hypersurface is equivalent to an
 isoparametric hypersurface in $S^n$ by a M\"{o}bius transformation or a Lie sphere transformation.
 In 1985, Cecil and Ryan \cite[p. 184]{CR7} conjectured
 (Conjecture \ref{cecil-ryan} below) that
 every compact, connected proper Dupin hypersurface
is equivalent to an isoparametric hypersurface in a sphere
by a Lie sphere transformation.  This paper
gives a survey of progress on this conjecture and related developments.

The theory of Dupin hypersurfaces is essentially the same in the two ambient spaces ${\bf R}^n$ and $S^n$,
since the Dupin property is preserved by stereographic projection $\tau :S^n - \{P\} \rightarrow {\bf R}^n$ with pole $P \in S^n$,
and its inverse map from ${\bf R}^n$ into $S^n$ (see, for example, \cite[pp. 132--151]{CR7},  
\cite[pp. 28--30]{CR8}).
In general, we will use whichever ambient space is most convenient to explain a certain concept.
Since much of the theory involves the relationship between Dupin hypersurfaces and
isoparametric hypersurfaces in spheres, we will first
formulate our definitions in terms of hypersurfaces in $S^n$.

Let $f:M \rightarrow S^n$
be an immersed hypersurface, and let $\xi$ be a locally defined
field of unit normals to $f(M)$.
A {\em curvature surface} of $M$ is a smooth submanifold $S\subset M$
such that for each point $x \in S$, the tangent space
$T_xS$ is equal to a principal space (i.e., an eigenspace) of the shape operator
$A$ of $M$ at $x$. This generalizes the classical notion of a line of curvature 
for a principal curvature of multiplicity one.

A hypersurface $M$ is said to be {\em Dupin} if:

\begin{enumerate}
\item[(a)] along each curvature surface, the corresponding principal curvature is constant.
\end{enumerate}
Furthermore, a Dupin hypersurface $M$ is called {\em proper Dupin} if, in addition to Condition (a),
the following condition is satisfied:
\begin{enumerate}
\item[(b)] the number $g$ of distinct principal curvatures is constant on
$M$.
\end{enumerate}

Clearly isoparametric hypersurfaces in $S^n$ are proper Dupin, and so are
those hypersurfaces in ${\bf R}^n$ obtained from isoparametric hypersurfaces in $S^n$ via
stereographic projection (see, for example, \cite[pp. 28--30]{CR8}).  
In particular, the well-known ring cyclides of Dupin in ${\bf R}^3$ are obtained
in this way from a standard product torus $S^1(r) \times S^1(s)$ in $S^3$, where $r^2+s^2=1$.

We now mention several basic facts about Dupin hypersurfaces 
(see, for example, \cite[pp. 9--35]{CR8} for proofs).  Let $f:M \rightarrow S^n$
be an immersed hypersurface, and let $\xi$ be a locally defined
field of unit normals to $f(M)$.
Using the Codazzi equation, one can show that if a curvature surface $S$ has dimension greater than one,
then the corresponding principal curvature is constant on $S$.  This is
not necessarily true on a curvature surface of dimension equal to one
(i.e., a line of curvature).

Second, Condition (b) is equivalent to requiring that each continuous
principal curvature function 
has constant multiplicity on $M$. Further, for any hypersurface $M$ in $S^n$, there exists a
dense open subset of $M$ on which 
the number of distinct principal
curvatures is locally constant (see, for example, Singley \cite{Sin}).

Next, it also follows from the Codazzi equation that if a continuous principal curvature function $\mu$ has constant
multiplicity $m$ on a connected open subset $U \subset M$, then $\mu$ is a smooth function, and
the distribution $T_{\mu}$ of principal
spaces corresponding to $\mu$ is a smooth foliation whose leaves are the curvature surfaces corresponding to 
$\mu$ on $U$. This principal curvature function
$\mu$ is constant along each of its curvature surfaces in $U$
if and only if these curvature surfaces are open subsets
of $m$-dimensional great or small spheres in $S^n$.  

Suppose that $\mu = \cot \theta$, for $0 < \theta < \pi$,
where $\theta$ is a smooth function on $U$.
The corresponding {\em focal map} $f_{\mu}$
which maps $x\in M$ to the focal point $f_{\mu}(x)$ is given by the formula,
\begin{equation}
\label{eq:focal-map}
f_{\mu}(x) = \cos \theta (x) \ f(x) + \sin \theta (x) \ \xi(x).
\end{equation}
The principal curvature $\mu$ also determines a second focal map
obtained by replacing $\theta$ by $\theta + \pi$ in equation (\ref{eq:focal-map}). 
The image of this second focal map is antipodal to the image of $f_{\mu}$.
The principal curvature function $\mu$ is constant along each of its curvature surfaces in $U$ if and only if
the focal map $f_{\mu}$ factors through an immersion of the
$(n-1-m)$-dimensional space of leaves $U/T_{\mu}$ into $S^n$, and so $f_{\mu}(U)$ is an $(n-1-m)$-dimensional
submanifold of $S^n$.

By definition, the {\em curvature sphere} $K_{\mu}(x)$ corresponding to
the principal curvature $\mu$ at a point $x \in U$ is the hypersphere in $S^n$ through $f(x)$ with one of its
centers at the focal point $f_{\mu}(x)$, and the other center at the focal point antipodal to $f_{\mu}(x)$.
Thus, $K_{\mu}(x)$ is tangent to $f(M)$ at $f(x)$.  It is easy to show that the 
principal curvature map $\mu$ is constant 
along each of its curvature surfaces if and only if the curvature sphere map $K_{\mu}$ is constant 
along each of these same curvature surfaces.  

In summary, on an open subset $U$ on which Condition (b) holds, Condition (a) is equivalent
to requiring that each curvature surface in each principal
foliation be an open subset of a metric sphere in $S^n$
of dimension equal to the multiplicity of the corresponding
principal curvature. Condition (a) is also equivalent to requiring that
along each curvature surface, the corresponding curvature sphere map is constant. Finally,  Condition (a)
is equivalent to requiring that for each principal curvature $\mu$, the image of the focal
map $f_{\mu}$ is a smooth submanifold of $S^n$ of codimension $m+1$, where $m$ is the multiplicity of $\mu$.

There exist many examples of Dupin hypersurfaces that are not proper Dupin, because the number of 
distinct principal curvatures is not constant on the hypersurface.
This also results in curvature surfaces that are not leaves of a 
principal foliation.  The following example due to Pinkall \cite{P4}
of a tube $M^3$ in ${\bf R}^4$ of constant radius over a torus of revolution 
$T^2 \subset {\bf R}^3 \subset {\bf R}^4$ illustrates these points.
This description of Pinkall's example is taken from the book \cite[p. 69]{Cec1}.

\begin{example}
\label{ex:3.4.6}
A  Dupin hypersurface that is not proper Dupin.\\
\noindent
{\rm Let $T^2$ be a torus of revolution 
in ${\bf R}^3$, and embed ${\bf R}^3$ into ${\bf R}^4 = {\bf R}^3 \times {\bf R}$.
Let $\eta$ be a field of unit normals to $T^2$ in ${\bf R}^3$.  Let $M^3$ be a 
tube of sufficiently small radius
$\varepsilon > 0$ around $T^2$ in ${\bf R}^4$, so that $M^3$ is a compact smooth embedded 
hypersurface in ${\bf R}^4$.
The normal space 
to $T^2$ in ${\bf R}^4$ at a point $x \in T^2$ is spanned by $\eta(x)$ and $e_4 = (0,0,0,1)$.
The shape operator $A_\eta$ of $T^2$ has two distinct principal curvatures at each point of $T^2$, while
the shape operator $A_{e_4}$ of $T^2$ is identically zero.  Thus the shape operator $A_\zeta$ for the normal
\begin{displaymath}
\zeta = \cos \theta \ \eta (x) + \sin \theta \ e_4
\end{displaymath}
at a point $x \in T^2$ is given by 
\begin{displaymath}
A_\zeta = \cos \theta \ A_{\eta(x)}.
\end{displaymath}
From the formulas for the 
principal curvatures of a tube
(see Cecil--Ryan \cite[p. 17]{CR8}), one finds that at all points of $M^3$ where
$x_4 \neq \pm \varepsilon$, there are three distinct principal curvatures of multiplicity one, which are constant 
along their corresponding lines of curvature
(curvature surfaces of dimension one). However, on the two tori,
$T^2 \times \{ \pm \varepsilon\}$, the principal curvature $\kappa = 0$ has multiplicity two.  These two tori
are curvature surfaces for this principal curvature, since the principal space corresponding to $\kappa$ is
tangent to each torus at every point.  $M^3$ is a Dupin hypersurface in ${\bf R}^4$, but it is not proper Dupin,
since the number of distinct principal curvatures is not constant on $M^3$.  The two tori 
$T^2 \times \{ \pm \varepsilon\}$ are curvature surfaces, but they are not leaves of a principal foliation on $M^3$.}
\end{example}

One consequence of the results mentioned above
is that proper Dupin hypersurfaces are algebraic, as is the case with isoparametric hypersurfaces,
as shown by M\"{u}nzner \cite{Mu}--\cite{Mu2}.
This result is most easily formulated for hypersurfaces in ${\bf R}^n$.  It states
that a connected proper Dupin hypersurface 
$f:M \rightarrow {\bf R}^n$ must be contained in a connected component of an 
irreducible algebraic subset of ${\bf R}^n$ of dimension $n-1$.  Pinkall \cite{P6} sent the author a letter in 1984 that contained a sketch of a proof of this result, but he did not publish a proof.  In 2008,
Cecil, Chi and Jensen \cite{CCJ3} used methods of real
algebraic geometry to give a proof of this result
based on Pinkall's sketch. The proof makes use of the various principal foliations whose leaves are open subsets of spheres to construct an analytic algebraic parametrization of a 
neighborhood of $f(x)$ for each point $x \in M$. 
In contrast to the situation for isoparametric hypersurfaces, however, a connected proper Dupin hypersurface 
in $S^n$ does not necessarily lie in a compact connected proper Dupin hypersurface,
as Example \ref{ex:3.4.6} illustrates.

The definition of Dupin can be extended to submanifolds of
codimension greater than one as follows.  Let
$\phi:V \rightarrow {\bf R}^n$ (or $S^n$) be a submanifold of
codimension greater than one, and let $B^{n-1}$ denote the
unit normal bundle of $\phi(V)$.  In this case, a {\em curvature
surface} (see Reckziegel \cite{Reck2}) is defined to be a connected submanifold $S \subset V$ for which
there is a parallel section $\eta : S \rightarrow B^{n-1}$
such that for each $x \in S$, the tangent space $T_xS$ is
equal to some smooth eigenspace of the shape operator
$A_{\eta}$. The submanifold $\phi(V)$ is said to be {\em Dupin}
if along each curvature surface, the corresponding principal
curvature of $A_{\eta}$ is constant. A Dupin submanifold is {\em proper Dupin} if the number
of distinct principal curvatures is constant
on the unit normal bundle $B^{n-1}$.  

Terng \cite{Te1} generalized the notion of an isoparametric hypersurface
to submanifolds of codimension greater than one, and she
proved that an isoparametric submanifold of codimension greater than one is always
Dupin, but it may not be proper Dupin \cite[pp. 464--469]{Te2}.
In related results, Pinkall \cite[p. 439]{P4} proved that every extrinsically 
symmetric submanifold of a real space form is Dupin.  Then Takeuchi
\cite{Tak} determined which of these are proper Dupin.

\section{Submanifolds in Lie sphere geometry}
\label{sec:2}
Many of the important results described in this paper concern finding conditions under which a compact 
proper Dupin hypersurface in $S^n$ is equivalent by a Lie sphere transformation to an isoparametric hypersurface
in $S^n$.  To make this precise, we now give a brief description of the method for studying hypersurfaces
in ${\bf R}^n$ or $S^n$ within the context of Lie sphere geometry
(introduced by Lie \cite{Lie}).  The reader is referred to the papers of
Pinkall \cite{P4}, Chern \cite{Chern}, Cecil and Chern \cite{CC1}--\cite{CC2}, Cecil and Jensen \cite{CJ2},
or the books of Cecil \cite{Cec1}, or
Jensen, Musso and Nicolodi \cite{Jensen-Musso-Nicolodi}, for more detail.

Lie sphere geometry is situated in real projective space ${\bf P}^n$,
so we now briefly review some concepts and notation from projective geometry.
We define an equivalence relation on ${\bf R}^{n+1} - \{0\}$ by setting
$x \simeq y$ if $x = ty$ for some nonzero real number $t$.  We denote
the equivalence class determined by a vector $x$ by $[x]$.  Projective
space ${\bf P}^n$ is the set of such equivalence classes, and it can
naturally be identified with the space of all lines through the origin
in ${\bf R}^{n+1}$.  The rectangular coordinates $(x_1, \ldots, x_{n+1})$
are called {\em homogeneous coordinates}
of the point $[x]$ in ${\bf P}^n$, and they
are only determined up to a nonzero scalar multiple.  

A {\em Lie sphere} in $S^n$ is an oriented hypersphere or a point sphere in $S^n$.
The set of all Lie spheres
is in bijective correspondence with the set of all points $[x] = [(x_1,\ldots,x_{n+3})]$
in projective space ${\bf P}^{n+2}$ that lie
on the quadric hypersurface $Q^{n+1}$ 
determined by the equation 
$\langle x, x \rangle = 0$, where
\begin{equation}
\label{Lie-metric}
\langle x, y \rangle = -x_1 y_1 + x_2 y_2 + \cdots +x_{n+2} y_{n+2} - x_{n+3} y_{n+3}
\end{equation}
is a bilinear form of signature $(n+1,2)$ on ${\bf R}^{n+3}$.  
We let  ${\bf R}^{n+3}_2$ denote ${\bf R}^{n+3}$ endowed with
the indefinite inner product (\ref{Lie-metric}).  The quadric 
$Q^{n+1} \subset {\bf P}^{n+2}$  is called the {\em Lie quadric}.

The specific correspondence is as follows.  We identify $S^n$ with
with the unit sphere in ${\bf R}^{n+1} \subset {\bf R}^{n+3}_2$, where ${\bf R}^{n+1}$ is spanned by the 
standard basis vectors $\{e_2,\ldots,e_{n+2}\}$ in ${\bf R}^{n+3}_2$.
Then the oriented hypersphere with center $p \in S^n$ and
signed radius $\rho$ corresponds to the point in $Q^{n+1}$ with homogeneous coordinates,
\begin{equation}
\label{eq:1.4.4}
(\cos \rho , p, \sin \rho ),
\end{equation}
where $ - \pi < \rho < \pi$.

We can designate the orientation of the sphere by the sign of
$\rho$ as follows.  A positive radius $\rho$ in 
(\ref{eq:1.4.4}) corresponds to the orientation of the sphere given by 
the field of unit normals which are tangent vectors to geodesics in $S^n$ 
going from $-p$ to $p$, and a negative radius corresponds
to the opposite orientation.
Each oriented sphere can be considered in two ways, with center $p$ and signed radius $\rho, - \pi < \rho < \pi$,
or with center $-p$ and the appropriate signed radius $\rho \pm \pi$.
Point spheres $p$ in $S^n$ correspond to those points $[(1, p, 0)]$ in $Q^{n+1}$ with radius $\rho = 0$.

Due to the signature of the metric $\langle \ ,\ \rangle$, the
Lie quadric $Q^{n+1}$ contains projective lines but no linear
subspaces of ${\bf P}^{n+2}$ of higher dimension (see, for example, \cite[p. 21]{Cec1}).  
A straightforward calculation
shows that if $[x]$ and $[y]$ are two points on the quadric, 
then the line $[x,y]$ 
lies on $Q^{n+1}$ if and only if $\langle x,y \rangle=0$.  Geometrically,
this condition means that the hyperspheres in $S^n$ corresponding
to $[x]$ and $[y]$ are in oriented contact, i.e., they are tangent to each other
and have the same orientation at the point of contact.
For a point sphere and an oriented sphere, oriented contact means that the point lies on the sphere.
The 1-parameter family
of Lie spheres  in $S^n$ corresponding to the points on a line on the Lie quadric is called a 
{\em parabolic pencil of spheres}.

If we wish to work in ${\bf R}^n$, the set of {\em Lie spheres} consists of all oriented hyperspheres, oriented hyperplanes, and point spheres in ${\bf R}^n \cup \{ \infty\}$.  As in the spherical case,
we can find a bijective correspondence between the set of all Lie spheres and the set of
points on $Q^{n+1}$, and the notions of oriented contact and parabolic pencils of
Lie spheres are defined in a natural way (see, 
for example, \cite[pp. 14--23]{Cec1}).

A {\em Lie sphere transformation} 
is a projective transformation of ${\bf P}^{n+2}$ which maps the Lie quadric $Q^{n+1}$
to itself.  In terms of the geometry of ${\bf R}^n$ (or $S^n$), 
a Lie sphere transformation maps Lie spheres to Lie spheres.
Furthermore, since a Lie sphere transformation maps lines on $Q^{n+1}$ to lines on $Q^{n+1}$, 
a Lie sphere transformation preserves oriented contact of Lie spheres  (see Pinkall \cite[p. 431]{P4}
or \cite[pp. 25--30]{Cec1}).

The group of Lie sphere transformations is
isomorphic to $O(n+1,2)/ \{ \pm I \}$, where $O(n+1,2)$
is the orthogonal group for the metric in equation (\ref{Lie-metric}).
A Lie sphere transformation that takes point spheres to point spheres is a {\em M\"{o}bius transformation}, i.e.,
it is induced by a conformal diffeomorphism of $S^n$, and the set of all M\"{o}bius transformations is a subgroup of the Lie sphere group.  
An example of a Lie sphere transformation that is not
a M\"{o}bius transformation is a parallel transformations $P_t$, which fixes the center of each Lie
sphere but adds $t$ to its signed radius (see \cite[pp. 25--49]{Cec1}).

The $(2n-1)$-dimensional manifold $\Lambda^{2n-1}$ of projective lines on the quadric $Q^{n+1}$ has a 
contact structure, i.e., a
$1$-form $\omega$ such that $\omega \wedge (d\omega)^{n-1}$ does not vanish on $\Lambda^{2n-1}$.  The condition $\omega = 0$ defines a codimension one distribution $D$ on $\Lambda^{2n-1}$ which has 
integral submanifolds
of dimension $n-1$, but none of higher dimension.  Such an integral  submanifold 
$\lambda: M^{n-1} \rightarrow \Lambda^{2n-1}$ of $D$ of dimension $n-1$ is called a 
{\em Legendre submanifold} (see \cite[pp. 51--64]{Cec1}).

An oriented hypersurface $f:M^{n-1} \rightarrow S^n$ with field of unit
normals $\xi :M^{n-1} \rightarrow S^n$ naturally induces
a Legendre submanifold $\lambda = [k_1, k_2]$, where 
\begin{equation}
k_1 = (1,f,0), \quad k_2 = (0,\xi ,1),
\end{equation}
in homogeneous coordinates.
For each $x \in M^{n-1}, [k_1(x)]$ is the unique point sphere and
$[k_2 (x)]$ is the unique great sphere in the parabolic pencil of spheres in $S^n$ corresponding to 
the points on the line $\lambda (x)$.
Similarly, an immersed submanifold $\phi :V \rightarrow S^n$ of codimension 
greater than one induces a Legendre submanifold 
whose domain is the bundle $B^{n-1}$ of unit normal
vectors to $\phi (V)$.  In each case, $\lambda$ is called the {\em Legendre lift} of the submanifold in $S^n$.
In a similar way, a submanifold of  ${\bf R}^n$ naturally induces a Legendre submanifold.

If $\beta$  is a Lie sphere transformation, then $\beta$ maps lines on
$Q^{n+1}$ to lines on $Q^{n+1}$, so it naturally
induces a map $\tilde{\beta }$ from 
$\Lambda ^{2n-1}$ to itself.  If $\lambda$ is a
Legendre submanifold, then one can show that $\tilde{\beta }\lambda$
is also a Legendre submanifold, which is denoted as
$\beta \lambda$ for short.  These two Legendre
submanifolds are said to be {\em Lie equivalent}.
We will also say that two submanifolds of $S^n$ or ${\bf R}^n$
are Lie equivalent, if their corresponding Legendre lifts
are Lie equivalent. 
 
If $\beta$ is a M\"{o}bius transformation, then the two Legendre
submanifolds  $\lambda$ and $\beta \lambda$ are said to be {\em M\"{o}bius equivalent}.
Finally, if $\beta$ is the parallel transformation $P_t$ and
$\lambda$ is the Legendre lift of an
oriented hypersurface $f:M \rightarrow S^n$, then
$P_t\lambda$ is the Legendre lift of the
parallel hypersurface $f_{-t}$ at oriented distance $-t$ from $f$ (see, for example, \cite[p.67]{Cec1}).

It is easy to generalize the definitions of Dupin and proper Dupin hypersurfaces in $S^n$ to the class of Legendre submanifolds in Lie sphere geometry.  We simply replace the notion of a principal curvature (which is not Lie invariant)
with the notion of a curvature sphere (which is Lie invariant).
We then say that a Legendre submanifold $\lambda: M^{n-1} \rightarrow \Lambda^{2n-1}$
is a {\em Dupin submanifold} if:

\begin{enumerate}
\item[(a)] along each curvature surface, the corresponding
curvature sphere map is constant.
\end{enumerate}
Furthermore, a Dupin submanifold $\lambda$ is called {\em proper Dupin} if, in addition 
to Condition (a), the following condition is satisfied:

\begin{enumerate}
\item[(b)] the number $g$ of distinct curvature spheres is constant on $M$.
\end{enumerate}

One can easily show that a Lie sphere transformation $\beta$ maps curvature spheres of $\lambda$ to curvature
spheres of $\beta \lambda$, and that Conditions (a) and (b) are preserved by $\beta$ (see \cite[pp.67--70]{Cec1}).  
Thus, both the Dupin and proper Dupin properties are invariant under Lie sphere transformations.

\section{Local constructions}
\label{sec:3}

In this section, we discuss Pinkall's  \cite{P4} method for constructing local examples of proper Dupin hypersurfaces in Euclidean space that have an arbitrary number of distinct principal curvatures with any given
multiplicities.  Specifically, these are constructions
for obtaining a Dupin hypersurface $W$ in ${\bf R}^{n+m}$
from a Dupin hypersurface $M$ in ${\bf R}^n$.  
We first describe these constructions in the case $m=1$ as follows.

Begin with a Dupin hypersurface
$M^{n-1}$ in ${\bf R}^n$ and then consider ${\bf R}^n$
as the linear subspace ${\bf R}^n \times \{ 0 \}$
in ${\bf R}^{n+1}$.  The following 
constructions yield a Dupin hypersurface $W^n$ in
${\bf R}^{n+1}$.
\begin{eqnarray}
\label{eq:5.97}
&i.& W^n\  {\rm is\ a\ cylinder}\ M^{n-1} \times {\bf R} \ {\rm in}\ {\bf R}^{n+1}.           \nonumber \\
&ii.& W^n\ {\rm is\ the\ hypersurface\ in}\ {\bf R}^{n+1}\ {\rm obtained\ by\ rotating}\  M^{n-1}\ \nonumber \\
&   & {\rm around\  an\  axis}\ {\bf R}^{n-1} \subset {\bf R}^n. \\
&iii.& W^n\ {\rm is\ a\ tube\ of\ constant\ radius\ around}\ M^{n-1} \ {\rm in}\ {\bf R}^{n+1}.\nonumber\\
&iv.& {\rm Project}\ M^{n-1}\ {\rm stereographically\ onto\ a\ hypersurface}\  V^{n-1} \subset S^n \nonumber\\
&   & {\rm in}\ {\bf R}^{n+1}. \  W^n \ {\rm is\ the\ cone\  over}\  V^{n-1} \ {\rm in}\ {\bf R}^{n+1}.\nonumber 
\end{eqnarray}

In general, these constructions introduce a new principal curvature 
of multiplicity one on $W^n$ which is constant along its lines 
of curvature.  The other principal curvatures are determined by the 
principal curvatures of $M^{n-1}$, and the Dupin property is preserved
for these principal curvatures.  Thus, 
if $M^{n-1}$ is a proper Dupin hypersurface in ${\bf R}^n$
with $g$ distinct principal curvatures, then in general, $W^n$ is a 
proper Dupin hypersurface in ${\bf R}^{n+1}$ with $g+1$ distinct principal curvatures.
However, this is not always the case, as there are cases where the number
of distinct principal curvatures of $W^n$ does not equal $g+1$ at some points, as we will discuss after the proof of
Theorem \ref{thm:4.1.1} below.   These constructions can be
generalized to produce a new principal curvature of multiplicity
$m$ by considering ${\bf R}^n$ as a subset of ${\bf R}^n \times
{\bf R}^m$ rather than ${\bf R}^n \times {\bf R}$.

Although Pinkall gave these four constructions, he showed \cite[p. 438]{P4} 
that the cone construction is redundant, since it
is Lie equivalent
to the tube construction (see also \cite[p.144]{Cec1}). Thus, we often work with just the cylinder, 
surface of revolution, and tube constructions, as in \cite{Cec1}.

Using these constructions, Pinkall \cite{P4} showed how
to produce a proper Dupin hypersurface in Euclidean space with an
arbitrary number of distinct principal curvatures, each with any given multiplicity as follows.

\begin{theorem}
\label{thm:4.1.1} 
(Pinkall, 1985)
Given positive integers $m_1,\ldots,m_g$ with 
\begin{displaymath}
m_1 + \cdots + m_g = n-1,
\end{displaymath}
there exists a proper Dupin hypersurface
in ${\bf R}^n$ with $g$ distinct principal curvatures having respective multiplicities $m_1,\ldots,m_g$.
\end{theorem}
\begin{proof}
The proof is by an inductive construction, which will be clear once we do the first few examples.  To begin, note
that a usual torus of revolution
in ${\bf R}^3$ is a proper Dupin hypersurface with two principal curvatures.  To
construct a proper Dupin hypersurface $M^3$ in ${\bf R}^4$ with three principal curvatures of multiplicity one,
begin with an open subset $U$ of a torus of revolution in ${\bf R}^3$ on which neither principal curvature vanishes.
Take $M^3$ to be the cylinder
$U \times {\bf R}$ in ${\bf R}^3 \times {\bf R} = {\bf R}^4$.  Then $M^3$ has three
distinct principal curvatures at each point, one of which is identically zero.  These are clearly constant along their
corresponding 1-dimensional curvature surfaces
(lines of curvature).

To get a proper Dupin hypersurface in ${\bf R}^5$ with three principal curvatures having respective multiplicities
$m_1 = m_2 = 1$, $m_3 =2$, one simply takes the cylinder
\begin{displaymath}
U \times {\bf R}^2 \subset {\bf R}^3 \times {\bf R}^2 = {\bf R}^5,
\end{displaymath}
where $U$ is the open subset of the torus defined above.
To obtain a proper Dupin hypersurface $M^4$ in ${\bf R}^5$ with four principal curvatures of multiplicity one, 
first invert the hypersurface
$M^3$ defined above in a 3-sphere in ${\bf R}^4$ chosen so that the image of $M^3$ contains an open subset $W^3$ on which no
principal curvature vanishes.  The hypersurface $W^3$ is proper Dupin, since the proper Dupin property is preserved by
M\"{o}bius transformations.  Now take $M^4$ to be the cylinder
$W^3 \times {\bf R}$ in ${\bf R}^4 \times {\bf R} = {\bf R}^5$. Then $M^4$ is
a proper Dupin hypersurface in ${\bf R}^5$ with four principal curvatures of multiplicity one.
\end{proof}

In general, there are problems in trying to produce compact proper Dupin hypersurfaces by using the
constructions in equation (\ref{eq:5.97}).  We now examine some of the problems involved
with the the cylinder, 
surface of revolution, and tube constructions individually (see \cite[pp. 127--141]{Cec1} for more details).

For the cylinder construction (\ref{eq:5.97}i),
the new principal curvature of $W^n$ is identically zero, while 
the other principal curvatures of $W^n$ are equal to those
of $M^{n-1}$.  Thus, if one of the principal curvatures $\mu$
of $M^{n-1}$ is zero at some points but not identically
zero, then the number of distinct principal curvatures 
is not constant on $W^n$, and so $W^n$ is Dupin but not proper Dupin.  

For the surface of revolution construction (\ref{eq:5.97}ii), if the focal point corresponding to a principal curvature
$\mu$ at a point $x$ of the profile submanifold $M^{n-1}$ lies on the axis of revolution ${\bf R}^{n-1}$, then
the principal curvature of $W^n$ at $x$ 
determined by $\mu$ is equal to the new principal curvature of $W^n$ resulting from
the surface of revolution construction.
Thus, if the focal point of $M^{n-1}$ corresponding to $\mu$ lies
on the axis of revolution for some but not all points of  $M^{n-1}$, then $W^n$ is not proper Dupin.

If $W^n$ is a tube (\ref{eq:5.97}iii) in ${\bf R}^{n+1}$ of radius
$\epsilon$ over $M^{n-1}$, then there are exactly two
distinct principal curvatures at the points in the set
$M^{n-1} \times \{ \pm \epsilon \}$ in $W^n$, regardless
of the number of distinct principal curvatures on $M^{n-1}$.
Thus, $W^n$ is not a proper Dupin hypersurface unless the original hypersurface $M^{n-1}$ is totally umbilic, i.e., it
has only one distinct principal curvature at each point.

Another problem with these constructions is that they may not yield
an immersed hypersurface in ${\bf R}^{n+1}$. In the tube construction, if the
radius of the tube is the reciprocal of one of the 
principal curvatures of $M^{n-1}$ at some point, then the constructed object has a singularity.  For the
surface of revolution construction, a singularity occurs
if the profile submanifold $M^{n-1}$ intersects the axis of revolution. 

Many of the issues mentioned in the preceding paragraphs can be resolved by working
in the context of Lie sphere geometry and considering
Legendre lifts of hypersurfaces in Euclidean space (see \cite[pp.127--141]{Cec1}).
In that context, a proper Dupin submanifold 
$\lambda : M^{n-1} \rightarrow \Lambda ^{2n-1}$ is said to be {\em reducible} if it is
is locally Lie equivalent to the Legendre lift of a hypersurface in ${\bf R}^n$ obtained by one of Pinkall's constructions in equation (\ref{eq:5.97}).

Pinkall \cite{P4} found a useful characterization of reducibility in the context of Lie sphere 
geometry when he proved that a proper Dupin submanifold
$\lambda : M^{n-1} \rightarrow \Lambda ^{2n-1}$ is reducible if and only if
the image of one its curvature sphere maps $K$ lies in a linear subspace of codimension two in ${\bf P}^{n+2}$
(see also \cite[pp. 141--148]{Cec1}).

One can obtain a reducible compact proper Dupin hypersurface with two principal curvatures 
by revolving a circle $C$ in ${\bf R}^3$ about an axis ${\bf R}^1 \subset {\bf R}^3$ that is disjoint from $C$ 
to obtain a torus of revolution.
However, Cecil, Chi and Jensen \cite{CCJ2} (see also \cite[pp. 146--147]{Cec1}) showed that 
every compact proper Dupin hypersurface with more than two principal curvatures is irreducible.
\noindent
\begin{theorem}
\label{thm:CCJ-2007}
(Cecil-Chi-Jensen, 2007)
If $M^{n-1} \subset {\bf R}^n$ is a compact, connected proper Dupin hypersurface with $g \geq 3$
principal curvatures, then $M^{n-1}$ is irreducible.
\end{theorem}

The proof uses known facts about the topology of a compact proper Dupin hypersurface
and the topology of a compact hypersurface obtained by one of Pinkall's constructions
(see \cite{CCJ2} or \cite[pp. 146--148]{Cec1} for a complete proof).\\

\begin{remark}
\label{rem:irreducibility}
{\rm From Theorem \ref{thm:CCJ-2007}, we see that 
one approach to obtaining classifications of compact proper Dupin hypersurfaces with more
than two principal curvatures is by assuming that the hypersurface is irreducible and 
then working locally in the context of Lie sphere geometry using the method of moving frames.  This approach has been used successfully in the papers of 
Pinkall \cite{P1}, \cite{P3}--\cite{P4}, Cecil and Chern \cite{CC2}, Cecil and Jensen \cite{CJ2}--\cite{CJ3}, 
and Cecil, Chi and Jensen \cite{CCJ2}.  We will discuss this in more detail in Section \ref{sec:4}
(see also \cite[pp. 168--190]{Cec1}).}
\end{remark}

\section{Isoparametric hypersurfaces}
\label{sec:3a}
In this section, we briefly review the theory of isoparametric hypersurfaces in real
space forms. This leads to many important examples of compact proper Dupin hypersurfaces.
Recall that an immersed hypersurface $M$ in a real space form,
${\bf R}^n$, $S^n$, or real hyperbolic space $H^n$,
is said to be {\em isoparametric} if it has constant principal curvatures.

A connected isoparametric hypersurface $M$ in ${\bf R}^n$ can have at most
two distinct principal curvatures, and $M$ must be an open 
subset of a hyperplane, hypersphere or a spherical cylinder
$S^k \times {\mathbf R}^{n-k-1}$.  This was shown by Levi--Civita \cite{Lev}
for $n=3$ and by B. Segre \cite{Seg} for arbitrary $n$. 

E. Cartan  \cite{Car2}--\cite{Car5} began the study of isoparametric hypersurfaces in the other space forms in a series of four papers in the 1930's.  In hyperbolic space
$H^n$, he showed that a connected isoparametric hypersurface can have at most two distinct
principal curvatures, and it is either totally umbilic or else an open subset of a standard 
product $S^k \times H^{n-k-1}$ in $H^n$
(see also Ryan \cite[pp. 252--253]{Ryan3} or Cecil-Ryan \cite[pp. 97--98]{CR8}).

In the sphere $S^n$, however, Cartan showed that there are many more possibilities.
He found examples of isoparametric hypersurfaces in $S^n$ with $1,2,3$ or 4 distinct principal curvatures, and he
classified isoparametric hypersurfaces in $S^n$ with $g \leq 3$ principal curvatures, as we now describe.

Let $M \subset S^n$ be a connected isoparametric hypersurface.
If $g=1$, then $M$ is totally umbilic, and it must be an open subset of a
great or small sphere.  If $g=2$, then $M$ must be an open subset of a standard product of two spheres,
\begin{displaymath}
S^p(r) \times S^{n-p-1}(s) \subset S^n, \quad r^2+s^2=1,
\end{displaymath}
where $1 \leq p \leq n-2$.

In the case $g=3$, Cartan \cite{Car3}
showed that all the principal curvatures must have the same multiplicity
$m=1,2,4$ or 8, and the isoparametric hypersurface must be an open subset of a tube of
constant radius 
over a standard embedding of a projective
plane ${\bf FP}^2$ into $S^{3m+1}$, 
where ${\bf F}$ is the division algebra
${\bf R}$, ${\bf C}$, ${\bf H}$ (quaternions),
${\bf O}$ (Cayley numbers), for $m=1,2,4,8,$ respectively (see also Cecil-Ryan \cite[pp. 151--155]{CR8}).  
Thus, up to congruence, there is only one such parallel family of isoparametric hypersurfaces for each value of $m$.

Cartan's theory was further developed by Nomizu \cite{Nom3}--\cite{Nom4},
Takagi and Takahashi \cite{TT}, Ozeki and Takeuchi \cite{OT}--\cite{OT2}, 
and by M\"{u}nzner \cite{Mu}--\cite{Mu2}, who proved the following fundamental result
(see also Chapter 3 of Cecil--Ryan
\cite{CR8} or the survey article by Thorbergsson \cite{Th6}).  Note that M\"{u}nzner's papers were published
in 1980--1981, although the first preprints of these papers appeared in 1973, and the results were used by other researchers shortly after that.

\begin{theorem}
\label{thm:Mu}
(M\"{u}nzner) The number $g$ of distinct principal curvatures of a connected isoparametric
hypersurface $M \subset S^n$ must be $1,2,3,4$ or $6$.
\end{theorem}

M\"{u}nzner first showed that every connected isoparametric hypersurface in $S^n$ is algebraic, and 
that it is 
an open subset of a unique compact, connected isoparametric hypersurface in $S^n$.
The proof of Theorem \ref{thm:Mu} then begins with the assumption that $M \subset S^n$ is a compact, connected isoparametric hypersurface.  M\"{u}nzner showed that
$M$ determines a family of parallel
isoparametric hypersurfaces together with two focal submanifolds, $M_+$ and $M_-$, and the sphere $S^n$ is the union of these parallel hypersurfaces and the two focal submanifolds.  
M\"{u}nzner also showed that all
of the parallel hypersurfaces and the two focal submanifolds are connected.
We will call such a family  of parallel isoparametric hypersurfaces together with the two focal submanifolds an
{\em isoparametric family} of hypersurfaces.

M\"{u}nzner's proof of Theorem \ref{thm:Mu}
involves a lengthy, delicate computation using the cohomology rings of $M$ and its two focal submanifolds in $S^n$.  The structure of these cohomology rings is determined by the topological situation 
that $M  \subset S^n$ divides $S^n$ into
two ball bundles over the two focal submanifolds of $M$ in $S^n$.
These two focal submanifolds, $M_+$ and $M_-$, lie in different components of the complement of $M$ in $S^n$.

This topological
situation has been used by various authors to determine the possible multiplicities of the principal curvatures of an isoparametric hypersurface with $g=4$ or $g=6$
principal curvatures.  This same topological situation also holds for compact, connected proper Dupin
hypersurfaces embedded in $S^n$, and thus the same restrictions on the number $g$ of distinct 
principal curvatures and their
multiplicities also hold on those hypersurfaces, as we will describe in Section \ref{sec:4}.

Much progress has been made on the study of isoparametric hypersurfaces in $S^n$ with $g=4$ or $g=6$ 
principal curvatures, as we will now describe.
In the case $g=4$, 
M\"{u}nzner proved that
the principal curvatures can have at most two
distinct multiplicities $m_1,m_2$. Ferus, Karcher and M\"{u}nzner
\cite{FKM} then used representations
of Clifford algebras to construct for every positive integer
$m_1$ an infinite series of isoparametric families of hypersurfaces
with four principal curvatures having respective multiplicities
$(m_1,m_2)$, where $m_2$ is nondecreasing and
unbounded in each series (see also \cite[pp. 95--112]{Cec1}, \cite[pp. 162--180]{CR8}).   This class of
{\it FKM-type} isoparametric hypersurfaces contains
all examples of isoparametric
hypersurfaces with four principal curvatures with the exception of two
homogeneous families, which have multiplicities $(2,2)$ and $(4,5)$.
For the FKM-type examples, one of the focal submanifolds is always a
{\em Clifford-Stiefel manifold}, as described by Pinkall and Thorbergsson \cite{PT1}
(see also \cite[p. 174]{CR8}).

Stolz \cite{Stolz} then proved that the multiplicities
of the principal curvatures of an isoparametric
hypersurface with four principal curvatures
must be the same as those of an isoparametric hypersurface of
FKM-type or one of the two homogeneous exceptions.
Cecil, Chi and Jensen \cite{CCJ1} next
showed that if the multiplicities satisfy $m_2 \geq 2 m_1 - 1$, then the
isoparametric hypersurface must be of FKM-type
(see also Immervoll \cite{Im} for a different proof of this result).
Taken together with known classification results of Takagi \cite{Takagi} for the case $m_1 = 1$,
and Ozeki and Takeuchi \cite{OT}--\cite{OT2} for the case $m_1 = 2$, this handles all
possible pairs of multiplicities except for four cases, the homogeneous pair $(4,5)$ and the FKM pairs
$(3,4), (6,9)$ and $(7,8)$.  

Following this, Chi handled the remaining open cases in a series of papers \cite{Chi2}--\cite{Chi4}
to complete the classification in the case $g=4$.  This showed that
every isoparametric hypersurface with four principal curvatures is either of FKM-type or 
else a homogeneous example with  multiplicities $(2,2)$ or $(4,5)$.

In the case of $g=6$ principal curvatures, M\"{u}nzner \cite{Mu}--\cite{Mu2} showed that all the principal curvatures
have the same multiplicity $m$, and Abresch \cite{Ab} showed that $m$ equals 1 or 2. 
Takagi and Takahashi \cite{TT} found homogeneous isoparametric families in both cases $m=1$ and $m=2$, and they showed that up to congruence, there is only one homogeneous family of isoparametric
hypersurfaces in each case.
 
Dorfmeister and Neher \cite{DN5} then proved in 1985 that
every isoparametric hypersurface $M^6 \subset S^7$ with $g=6$ principal curvatures of multiplicity $m=1$
is homogeneous. Their proof is very algebraic in nature, and later Miyaoka
\cite{Mi10} and Siffert \cite{Siffert1} gave alternative approaches to the proof in the case  $m=1$.

Regarding the case $m=2$, Miyaoka \cite{Mi11}
published a proof in 2013 that every isoparametric hypersurface with $g=6$ principal curvatures of multiplicity $m=2$
is homogeneous (see also the errata \cite{Mi12}).  
The errata pertain to an error in Miyaoka's original proof that
was pointed out by Abresch and Siffert.
(See also the papers of
Siffert \cite{Siffert1}--\cite{Siffert2} for more on the case $g=6$.)

\begin{remark}[Chern conjecture for isoparametric hypersurfaces]
\label{rem:Chern-conjecture}

\noindent
{\rm An active research problem in the area of isoparametric hypersurfaces is
the Chern conjecture for isoparametric hypersurfaces, which states that every closed minimal hypersurface 
immersed into the sphere with constant scalar curvature is isoparametric.  See
the papers of Scherfner and Weiss \cite{S-W}, Scherfner, Weiss and Yau \cite{S-W-Y}, 
and Ge and Tang \cite{Ge-Tang-2012},  
for information on this conjecture and its generalizations.}
\end{remark}

\begin{remark}[Real hypersurfaces in complex space forms]
\label{rem:complex-case}

\noindent
{\rm The study of real hypersurfaces with constant principal curvatures in complex projective space ${\bf CP}^n$ and complex hyperbolic space ${\bf CH}^n$ began at approximately the same time as M\"{u}nzner's  work on
isoparametric hypersurfaces in spheres, and it is an active area
of research.  An important early
work was Takagi's classification \cite{takagi1973} in 1973 of homogeneous real hypersurfaces
in ${\bf CP}^n$.  These hypersurfaces necessarily have constant principal curvatures, and they serve
as model spaces for many subsequent classification theorems.  Later Montiel \cite{montiel}
provided a similar list of standard examples in complex hyperbolic space ${\bf CH}^n$. 
This subject is described in great detail in Chapters 6--9 of the book of Cecil and Ryan
\cite{CR8}.}
\end{remark}

\section{Compact proper Dupin hypersurfaces}
\label{sec:4}
In this section, we discuss results on the classification of compact proper Dupin hypersurfaces.  
Of course, compact isoparametric hypersurfaces in $S^n$
have constant principal curvatures, so they are obviously compact
proper Dupin hypersurfaces in $S^n$. 

The images of compact isoparametric hypersurfaces under stereographic projection
from $S^n - \{P\}$ to ${\bf R}^n$ are compact proper Dupin hypersurfaces in ${\bf R}^n$
(see, for example, \cite[pp. 28--30]{CR8}). In addition,
since the proper Dupin property is invariant under Lie sphere transformations, 
any compact proper hypersurface in $S^n$ or ${\bf R}^n$ that is Lie equivalent to a
compact isoparametric hypersurface
in $S^n$ is a compact proper Dupin hypersurface.  This gives a large class of interesting examples of compact proper Dupin hypersurfaces.  Since the 1970's, an important research topic in this area has been to what extent there are any other examples.

Following M\"{u}nzner's work, Thorbergsson \cite{Th1} proved the following theorem which shows that 
the restriction in Theorem \ref{thm:Mu} on the number $g$ of 
distinct principal curvatures also holds for compact proper Dupin hypersurfaces embedded in $S^n$.  This is in stark contrast to Pinkall's Theorem \ref{thm:4.1.1} 
which states that
there are no restrictions on the number of distinct principal curvatures or their multiplicities for
non-compact proper Dupin hypersurfaces.

\begin{theorem}
\label{thm:thorbergsson}
(Thorbergsson, 1983) The number $g$ of distinct principal curvatures of a compact, connected
proper Dupin hypersurface
$M \subset S^n$ must be $1,2,3,4$ or $6$.
\end{theorem}

In proving this theorem,
Thorbergsson first shows that a compact, connected proper Dupin hypersurface $M \subset S^n$
must be tautly embedded,
that is, every nondegenerate spherical distance function $L_p (x) = d(p,x)^2$,  for  $p \in S^n$,
has the minimum number of critical points required by the Morse inequalities on $M$.
Thorbergsson then uses the fact that $M$ is tautly embedded in $S^n$ to show that $M$ divides $S^n$ into
two ball bundles over the first focal submanifolds, $M_+$ and $M_-$, on either side of $M$ in $S^n$.
This gives the same topological situation as in the isoparametric case,
and the result then follows from M\"{u}nzner's  proof of Theorem \ref{thm:Mu} above.

Using Thorbergsson's results,
we can formulate the relationship between the taut and proper Dupin properties as follows.
This theorem was first proven for surfaces in $S^3$ or ${\bf R}^3$ by Banchoff \cite{Ban1}.

\begin{theorem}
\label{thm:taut-dupin}
Let  $M \subset S^n$ be a compact, connected hypersurface on which the number $g$ of distinct principal curvatures is constant.  Then $M$ is taut if and only if $M$ is proper Dupin.
\end{theorem}

\begin{proof}
As noted above, Thorbergsson \cite{Th1} proved that a 
compact, connected proper Dupin hypersurface $M \subset S^n$ must be tautly embedded.  Conversely,
Pinkall \cite{P5} and Miyaoka \cite{Mi2} (for hypersurfaces) independently proved that a taut submanifold 
embedded in $S^n$ is Dupin. Thus, a taut hypersurface $M$ is proper Dupin
if $g$ is constant on $M$ (see 
\cite[pp. 65--74]{CR8} for more discussion on the relationship between taut and Dupin submanifolds).
\end{proof}

The topological situation that $M$ divides $S^n$ into
two ball bundles over the first focal submanifolds, $M_+$ and $M_-$, on either side of $M$ in $S^n$
leads to important restrictions on the multiplicities of the principal 
curvatures of compact proper Dupin hypersurfaces, due to
Stolz \cite{Stolz} for $g=4$, and to Grove and Halperin \cite{GH} for $g=6$.
These restrictions were obtained by using advanced topological considerations in each case, and they show that 
the multiplicities of the principal curvatures of
a compact proper Dupin hypersurface embedded in $S^n$ must be the same as the multiplicities 
of the principal curvatures of some isoparametric hypersurface in $S^n$.

Grove and Halperin \cite{GH} also gave a list of the integral homology of all compact
proper Dupin hypersurfaces, and Fang \cite{Fang2} found results on the topology of compact proper Dupin hypersurfaces with $g=6$ principal curvatures.

In 1985, it was known that every compact, connected proper Dupin hypersurface 
$M \subset S^n$ (or ${\bf R}^n$) with
$g= 1,2$ or 3 principal curvatures is Lie equivalent to an isoparametric hypersurface in $S^n$.
At that time, every other known example
of a compact, connected proper Dupin hypersurface in $S^n$ 
was also Lie equivalent to an isoparametric hypersurface in $S^n$. 
This together with Thorbergsson's Theorem \ref{thm:thorbergsson} above led to the following
conjecture by Cecil and Ryan \cite[p. 184]{CR7} (which we have rephrased slightly).

\begin{conjecture} 
\label{cecil-ryan} 
(Cecil-Ryan, 1985) Every compact, connected proper Dupin hypersurface $M \subset S^n$ $($or ${\bf R}^n)$
is Lie equivalent to an isoparametric hypersurface in $S^n$.
\end{conjecture}

We now discuss the current status of this conjecture for each of the values
$g = 1,2,3,4,6$, respectively.  After that, we will give more detail about the individual cases of $g$.\\

\noindent
$g=1$: Conjecture is true. $M$ is totally umbilic, and so $M = S^{n-1} \subset S^n$, a metric hypersphere.  
Thus, $M$ is isoparametric.\\

\noindent
$g=2$: Conjecture is true (Cecil-Ryan \cite{CR2}, 1978). $M$ is a cyclide of Dupin.  
Specifically, $M$ is M\"{o}bius equivalent to an isoparametric hypersurface
\begin{displaymath}
S^p(r) \times S^q(s) \subset S^n, \quad r^2+s^2=1,
\end{displaymath}
where $p + q = n-1$.\\

\noindent
$g=3$: Conjecture is true (Miyaoka \cite{Mi1}, 1984). $M$ is Lie equivalent (but not necessarily 
M\"{o}bius equivalent) to an isoparametric hypersurface $W^{3m}$ in $S^{3m+1}$, for $m = 1,2,4$ or 8.\\

\noindent
$g=4$: Conjecture is false. Counterexamples due to:\\

\noindent
(a) Pinkall-Thorbergsson \cite{PT1} (1989): these counterexamples
are certain deformations of FKM-type \cite{FKM} isoparametric hypersurfaces.\\

\noindent
(b)  Miyaoka-Ozawa \cite{MO} (1989): these counterexamples are hypersurfaces of the form $M^6 = h^{-1}(W^3) \subset S^7$,
where $h:S^7 \rightarrow S^4$ is the Hopf fibration, and $W^3 \subset S^4$
is a compact proper Dupin hypersurface with 2 principal curvatures that is not isoparametric.\\

\noindent
$g=6$: Conjecture is false. Counterexamples due to:\\

\noindent
Miyaoka-Ozawa \cite{MO} (1989): these counterexamples are also hypersurfaces of the form
$M^6 = h^{-1}(W^3) \subset S^7$,
where $h:S^7 \rightarrow S^4$ is the Hopf fibration, and $W^3 \subset S^4$
is a compact proper Dupin hypersurface with 3 principal curvatures that is not isoparametric
(same construction as above).\\

We now make some comments about the individual cases based on the value of $g$.
The case $g = 1$ is simply the well-known case of totally umbilic
hypersurfaces, and thus $M$ is a great or small hypersphere in $S^n$.  

In the case $g=2$, a connected proper Dupin hypersurface with two principal curvatures
having respective multiplicities $p$ and $q$ is called a {\em cyclide of Dupin of characteristic $(p,q)$}.  Cecil and Ryan 
\cite{CR2} showed that a compact, connected
cyclide of characteristic $(p,q)$ is M\"{o}bius equivalent to a standard product of spheres
(which is an isoparametric hypersurface),
\begin{equation}
\label{eq:cyclide}
S^p (r) \times S^q (s) \subset S^n (1) \subset {\bf R}^{n+1}, \quad r^2 + s^2 = 1,
\end{equation}
where $p + q = n-1$.
The proof of Cecil and Ryan \cite{CR2} uses the assumption of the compactness of $M$
in an essential way. Later, working in the context of Lie sphere geometry,
Pinkall \cite{P4} gave a local classification of the cyclides of Dupin by showing
that every connected cyclide of characteristic $(p,q)$  is contained in a unique compact, connected
cyclide of characteristic $(p,q)$, and that any two compact, connected cyclides of characteristic $(p,q)$
are Lie equivalent to each other (see also \cite[pp. 148--159]{Cec1}).

In the case $g=3$, Miyaoka \cite{Mi1} proved that $M$ is Lie equivalent (although not necessarily
M\"{o}bius equivalent) to an isoparametric hypersurface.
Later Cecil, Chi and Jensen \cite{CCJ2} gave a different proof of this result, based on the local
theorem of Cecil and Jensen \cite{CJ2} which states that a connected irreducible proper Dupin hypersurface
in $S^n$ with three principal curvatures must be Lie equivalent to an isoparametric hypersurface in $S^n$.

As mentioned in Section \ref{sec:3a}, 
Cartan \cite{Car3} had shown earlier that an isoparametric hypersurface with $g=3$
principal curvatures is a tube over a standard
embedding of a projective plane ${\bf FP}^2$,
for ${\bf F} = {\bf R}, {\bf C}, {\bf H}$ (quaternions) or ${\bf O}$ (Cayley numbers)
in $S^4, S^7, S^{13}$
and $S^{25}$, respectively (see also \cite[pp. 151--155]{CR8}). 

The case $g=4$ resisted all
attempts at solution for several years and finally in papers published in 1989, 
counterexamples to the conjecture were constructed 
independently by Pinkall and Thorbergsson \cite{PT1},
and by Miyaoka and Ozawa \cite{MO}. These two constructions lead to different types of counterexamples, and
the method of Miyaoka and Ozawa also 
yields counterexamples to the conjecture with $g=6$ principal curvatures.  

In both constructions, a fundamental
Lie invariant, the Lie curvature introduced by Miyaoka \cite{Mi3},
was used to show that the examples are not Lie equivalent to
an isoparametric hypersurface.  Specifically, if $M$ is a proper Dupin hypersurface 
with four principal curvatures,
then the {\em Lie curvature} $\psi$ is defined to be
the cross-ratio of these principal curvatures.  If $M$ has six principal curvatures, then the Lie curvatures are the cross-ratios of the principal curvatures taken four at a time.

Viewed in the context of projective geometry, at each point $x \in M$, a Lie curvature
is the cross-ratio of four points along a projective line on $Q^{n+1}$ corresponding
to four curvature spheres of $M$ at $x$.  
Since a Lie sphere transformation maps curvature spheres to curvature
spheres and preserves cross-ratios, a Lie curvature is invariant under Lie sphere transformations
(see  \cite[pp. 72--82]{Cec1} for more detail).  

From the work of M\"{u}nzner \cite{Mu}--\cite{Mu2}, 
it is easy to show that in the case $g=4$, the 
Lie curvature $\psi$ has the constant value $1/2$ (when the principal curvatures are
appropriately ordered) on an isoparametric hypersurface.  For
the counterexamples to the conjecture, it was shown that $\psi \neq 1/2$ at some points, and therefore the examples cannot be Lie equivalent to an isoparametric hypersurface. 

The examples of Pinkall and Thorbergsson  are obtained by 
taking certain deformations of the isoparametric hypersurfaces
of FKM-type constructed by  Ferus, Karcher and M\"{u}nzner \cite{FKM} using representations of Clifford algebras.  
Pinkall and Thorbergsson proved that their examples are not Lie
equivalent to an isoparametric hypersurface by showing that the Lie curvature does not have the
constant value $\psi = 1/2$, 
as required for a hypersurface with $g=4$ that is Lie equivalent 
to an isoparametric hypersurface.  Using their methods, one can also show directly
that the Lie curvature is not constant on their examples (see \cite[pp. 309--314]{CR8}).

The construction of counterexamples to  Conjecture \ref{cecil-ryan} due to Miyaoka and Ozawa \cite{MO} (see also
\cite[pp. 117--123]{Cec1}) is based on the Hopf fibration
$h:S^7 \rightarrow S^4$.  
Miyaoka and Ozawa first show that if $W^3$ is a taut
compact submanifold of $S^4$, then $M = h^{-1}(W^3)$ is a taut compact submanifold of
$S^7$.  

Using this and the fact that tautness is equivalent to proper Dupin
for a compact, connected hypersurface in $S^n$ on which the number $g$ of distinct principal 
curvatures is constant (Theorem \ref{thm:taut-dupin}),
they show that if $W^3$ is a proper Dupin hypersurface in
$S^4$ with $g$ distinct principal curvatures, then
$h^{-1}(W^3)$ is a proper Dupin hypersurface in $S^7$ with $2g$ principal
curvatures.  

To complete the argument, they show that if a compact, connected 
hypersurface $W^3 \subset S^4$ is proper Dupin
but not isoparametric, then the Lie curvatures of
$h^{-1}(W^3)$ are not constant, and therefore $h^{-1}(W^3)$ is not Lie
equivalent to an isoparametric hypersurface in $S^7$.  For $g=2$ or 3, this gives
a compact proper Dupin hypersurface $M = h^{-1}(W^3)$ in $S^7$ with $g=4$ or 6, respectively,
that is not Lie equivalent to an isoparametric hypersurface.

As we have seen, all of the hypersurfaces described above are shown to be 
counterexamples to Conjecture \ref{cecil-ryan} by proving that they do not have
constant Lie curvatures.  This led to a revision of Conjecture \ref{cecil-ryan} by Cecil, Chi and Jensen
\cite[p. 52]{CCJ4} in 2007 that contains the additional assumption of constant Lie curvatures.
This revised conjecture
is still an open problem, although it  has been shown to be true in some cases,
which we will describe after stating the conjecture.

\begin{conjecture}
\label{revised-conjecture}
(Cecil-Chi-Jensen, 2007)
Every compact, connected proper Dupin hypersurface in $S^n$ with four or six principal curvatures
and constant Lie curvatures is Lie equivalent to an isoparametric hypersurface.
\end{conjecture}

In 1989, Miyaoka \cite{Mi3}--\cite{Mi4} showed that if some additional assumptions are made regarding the intersections of the leaves of the various principal foliations, then this revised conjecture is true in both cases
$g=4$ and 6.  
Thus far, however, it has not been proven that Miyaoka's additional assumptions are satisfied in general.
 
Cecil, Chi and Jensen \cite{CCJ2} 
made progress on the revised conjecture in the case $g=4$ by using the
fact that compactness implies irreducibility for a proper Dupin hypersurface
with $g \geq 3$ (see Theorem \ref{thm:CCJ-2007}), 
and then working locally with irreducible proper hypersurfaces in the context 
of Lie sphere geometry, as described earlier in Remark \ref{rem:irreducibility}.  We now mention some notable facts in the case $g=4$.\\

\noindent
Case $g = 4$: 
There is only one Lie curvature,
\begin{equation}
\label{eq:lie-curv}
\psi = \frac{(\mu_1 -\mu_2)(\mu_4 - \mu_3)}{(\mu_1 -\mu_3)(\mu_4 - \mu_2)},
\end{equation}
when we fix the order of the principal curvatures of $M$ to be,
\begin{equation}
\label{eq:pc-order}
\mu_1 < \mu_2 < \mu_3 < \mu_4.
\end{equation}
For an isoparametric
hypersurface with four principal curvatures ordered as in equation (\ref{eq:pc-order}),
M\"{u}nzner's results \cite{Mu}--\cite{Mu2} imply that
the Lie curvature
$\psi = 1/2$, and
the multiplicities satisfy $m_1 = m_3$, $m_2 = m_4$.
Furthermore, if $M \subset S^n$ is a compact, connected proper Dupin hypersurface 
with $g=4$, then the multiplicities of the principal curvatures must be the same as those of an isoparametric hypersurface by the work of Stolz \cite{Stolz}, so they satisfy $m_1 = m_3$, $m_2 = m_4$.

Cecil-Chi-Jensen \cite{CCJ2} proved the following 
local classification of irreducible proper Dupin hypersurfaces with four principal curvatures and constant Lie curvature $\psi = 1/2$. In the case where all the multiplicities equal one, this theorem was first proven
by Cecil and Jensen \cite{CJ3}.

\begin{theorem}
\label{CCJ} 
(Cecil-Chi-Jensen, 2007)
Let $M \subset S^n$ be a connected irreducible proper Dupin hypersurface with four principal curvatures 
ordered as in equation (\ref{eq:pc-order}) having multiplicities,
\begin{equation}
\label{eq:restricted}
m_1 = m_3 \geq 1, \quad m_2 = m_4 =1,
\end{equation}
and constant Lie curvature $\psi = 1/2$. Then $M$ is Lie equivalent to an
isoparametric hypersurface in $S^n$.
\end{theorem}

By Theorem \ref{thm:CCJ-2007} above, we know that
compactness implies irreducibility for proper Dupin hypersurfaces with more than two principal curvatures.
Furthermore, Miyaoka \cite{Mi3} proved that if $\psi$ is constant on a compact proper Dupin hypersurface $M \subset S^n$ with $g=4$, then $\psi = 1/2$ on $M$, when
the principal curvatures are ordered as in equation (\ref{eq:pc-order}).  
As a consequence, we get the following corollary of
Theorem \ref{CCJ}.

\begin{corollary}
Let $M \subset S^n$ be a compact, connected proper Dupin hypersurface with 
four principal curvatures having multiplicities
\begin{displaymath}
m_1 = m_3 \geq 1, \quad m_2 = m_4 =1,
\end{displaymath}
and constant Lie curvature $\psi$.  Then $M$ is Lie equivalent to an isoparametric hypersurface in $S^n$.
\end{corollary}

\noindent
{\bf Question:}  What about the general case where $m_2$ is also allowed to be greater than one, i.e.,
\begin{equation}
\label{eq:general}
m_1 = m_3 \geq 1,\quad  \ m_2 = m_4 \geq 1,
\end{equation}
and constant Lie curvature $\psi$?\\

\noindent
Regarding this question, we note that
the local proof of Theorem \ref{CCJ} uses the method of moving frames, and it involves a large system
of equations that contains certain sums if some $m_i$ is greater than one, but no 
corresponding sums if all $m_i$ equal one.  These sums make the calculations significantly more difficult.
So far this method has not led to a proof in the general case (\ref{eq:general}),
although Cecil, Chi and Jensen were able to handle the case (\ref{eq:restricted}),
where $m_1$ is allowed to be greater than one, but $m_2$ is restricted 
to equal one.
Key elements in the proof are the Lie geometric criteria for reducibility due to Pinkall \cite{P4} (see also
\cite[pp. 141--148]{Cec1}), and for Lie equivalence to
an isoparametric hypersurface due to Cecil \cite[p. 77]{Cec1}.\\

\noindent
Case $g = 6$:
Grove and Halperin (1987) proved that if $M \subset S^n$ is a compact proper Dupin hypersurface
with $g=6$ principal curvatures, then all the principal curvatures must have the same multiplicity $m$, and
$m = 1$ or 2, as is the case for an isoparametric hypersurface, as shown by Abresch \cite{Ab}.
They also proved other topological 
results about compact proper Dupin hypersurfaces that support Conjecture \ref{revised-conjecture} in the case $g=6$.

As mentioned above,
Miyaoka \cite{Mi4} showed that if some additional assumptions are made regarding the intersections of the leaves of the various principal foliations, then Conjecture \ref{revised-conjecture} is true in the case $g=6$.  
However, it has not been proven that Miyaoka's additional assumptions are satisfied in general,
and so Conjecture \ref{revised-conjecture} remains as an open problem in the case $g=6$.

\noindent Thomas E. Cecil

\noindent Department of Mathematics and Computer Science

\noindent College of the Holy Cross

\noindent Worcester, MA 01610

\noindent email: tcecil@holycross.edu\\

\end{document}